\documentclass[a4paper,12pt]{amsart}

\usepackage{lmodern}

\usepackage[T1]{fontenc}
\usepackage[utf8]{inputenc}
\usepackage[english]{babel}

\usepackage{amssymb,amsfonts,amsthm,amsmath,latexsym}
\usepackage{enumerate, caption, array, esint}
\usepackage{hyperref, color, tikz, pgfplots, subcaption}

\usepackage[margin=1.1in]{geometry}

\theoremstyle{plain}
\newtheorem{theorem}{Theorem}

\newtheorem{lemma}{Lemma}

\theoremstyle{definition}

\theoremstyle{remark}
\newtheorem*{remark}{Remark}

\newcommand\numberthis{\stepcounter{equation}\tag{\theequation}}

\newcommand{\ssum}[1]{\sum_{\substack{#1}}}

\newcommand{\e}{{\rm e}}
\newcommand{\dd}{{\rm d}}

\newcommand{\eps}{{\varepsilon}}

\newcommand{\C}{{\mathbb C}}
\newcommand{\R}{{\mathbb R}}

\newcommand{\Z}{{\mathbb Z}}
\newcommand{\N}{{\mathbb N}}
\newcommand{\1}{{\mathbf 1}}

\newcommand{\smatrix}[1]{\left(\begin{smallmatrix}#1\end{smallmatrix}\right)}
\newcommand{\ppmatrix}[1]{\begin{pmatrix}#1\end{pmatrix}}

\newcommand{\abs}[1]{\left|#1\right|}

\numberwithin{equation}{section}

\title{Two arithmetic applications of perturbations of composition operators}

\date{\today}

\author{S. Bettin}
\address{SB: DIMA - Dipartimento di Matematica, Via Dodecaneso, 35, 16146 Genova, Italy}
\email{bettin@dima.unige.it}

\author{S. Drappeau}
\address{SD: Aix Marseille Universit\'e, CNRS, Centrale Marseille, I2M UMR 7373, 13453 Marseille, France}
\email{sary-aurelien.drappeau@univ-amu.fr}



\dedicatory{À la mémoire de Christian Mauduit}

\begin{document}

\begin{abstract}
  We estimate the spectral radius of perturbations of a particular family of composition operators, in a setting where the usual choices of norms do not account for the typical size of the perturbation. We apply this to estimate the growth rate of large moments of a Thue-Morse generating function and of the Stern sequence. This answers in particular a question of Mauduit, Montgomery and Rivat (2018).
\end{abstract}

\maketitle

\section{Introduction}

The present note is concerned with a case of asymptotic perturbation of a linear operator, which is a widely studied subject; we refer to the monograph~\cite{Kato} and to the recent work~\cite{Kloeckner2019} for references. There are well-understood general results which deal with the behaviour of the spectrum of the perturbation~$T+\eps$ of a bounded linear operator~$T$, granted one can find a norm with respect to which~$\eps$ can indeed be considered a perturbation.

In the recent works~\cite{BDS, MauduitEtAl2018}, instances of this question arose which do not fall in the scope of the general analysis, the reason being that the natural norms one has do not account for the true expected magnitude of the perturbation. The purpose of this note is to present an alternate argument, which relies on an ad-hoc construction but allows to answer completely the questions in~\cite{BDS,MauduitEtAl2018}. We begin by a discussion of the two arithmetic applications we are considering.

\subsection{Moments of a Thue-Morse generating function}

In this section only, for all~$m\in\N$, we let~$t(m) \in \{\pm 1\}$ denote the parity of the sum of digits of~$m$ in base~$2$, so that~$(t(m))_{m\geq 0}$ is the celebrated Prouhet-Thue-Morse sequence~\cite{AlloucheShallit1999}. For all~$n\in\N$, we let~$T_n:\R/\Z \to\C$ be defined as
$$ T_n(x) = \prod_{0\leq r<n}(1-\e(2^r x)) = \sum_{0\leq m < 2^n} t(m) \e(mx). $$
In~\cite{MauduitEtAl2018}, the authors study the moments
$$ M_k(n) := \int_0^1 \abs{T_n(x)}^{2k} \dd x,\qquad k\in\N. $$
Upper-bounds on~$M_k(n)$ are an important ingredients on works on the level of distribution of the Thue-Morse sequence, in particular in~\cite{FouvryMauduit1996,MauduitRivat2010} where estimates of~$M_{1/2}(n) = \|T_n\|_1$ and~$\lim_{k\to\infty} M_k(n)^{1/(2k)} = \|T_n\|_\infty$ are used to obtain asymptotic formulas for the number of integers with multiplicative constraints (primes or almost-primes) having a predetermined parity of their sum-of-digits modulo~$2$.

In~\cite{MauduitEtAl2018}, the authors show that the sequence~$(M_k(n))_{n\geq 0}$ satisfies a linear recurrence equation, and they deduce, for each~$k>0$ the existence of constants~$C_k>0$ and~$\varrho_k >0$ such that
\begin{equation}
  M_k(n) \sim C_k \varrho_k^n \qquad (n\to+\infty).\label{eq:asymp-mkn}
\end{equation}
The behaviour of the constant~$\varrho_k$ as~$k\to+\infty$ was left as an open question in~\cite{MauduitEtAl2018}. The authors conjectured that~$\varrho_k = \tfrac12 3^k (1 + O(k^{-2}))$ for~$k\geq 1$. Towards this estimate, they show the upper-bound~$\varrho_k \leq \tfrac12(3^k + 4^{2k/3})$.

Using Theorem~\ref{thm:main} below we are able to prove this conjecture, isolating also a secondary term of size exponentially smaller.
\begin{theorem}\label{th:MMR}
  For~$\delta_1 = \prod_{n\geq 1} \frac{2}{\sqrt{3}}\sin(\frac{\pi}3(1 + \frac{(-1)^n}{2^n})) = 0.6027\dotsb$ and~$\eta = 0.506$, we have
  $$ \varrho_k = \tfrac12 3^k\big(1 + \delta_1^{2k} + O(\eta^{2k})\big). $$
\end{theorem}

\subsection{Moments of the Stern sequence}

Our second application concerns the Stern sequence $(s(n))_{n\in\N_{>0}}$, defined by~$s(1)$ and the recursion formula
$$ s(2n) = s(n), \qquad s(2n+1) = s(n)+s(n+1). $$
This sequence has been widely studied due to its links with Farey fractions and enumeration of the rationals~\cite{KesseboehmerEtAl2016}, Automatic sequences~\cite{AlloucheShallit2003}, or the Minkowski function and the thermodynamic formalism of the Farey map~\cite{PrellbergSlawny1992,Dodds,KesseboehmerStratmann2007,BDS}.

For all~$\tau\in\C$ and~$N\in\N_{>0}$, define the moment sequence
$$ M_\tau(N) := \sum_{2^N < n \leq 2^{N+1}} s(n)^\tau. $$
In~\cite{BDS}, the asymptotic estimation of~$M_\tau(N)$ as~$N\to\infty$ for~$\tau$ in a neighborhood of~$0$ led to a central limit theorem for the values~$\log s(n)$. The asymptotic behaviour of~$M_\tau(N)$ for~$\tau$ away from~$0$ is an interesting question. Let us focus on large integer moments,~$\tau=k\in\N$. It is not difficult to show, in analogy with~\eqref{eq:asymp-mkn}, that the sequence~$(M_\tau(N))_{N\geq 0}$ satisfies a linear recurrence equation, from which we deduce the following statement, to be proven in Section~\ref{sec:proof-theor-stern} below: for all~$k\in\N$, there are constants~$D_k>0$ and~$\sigma_k>0$ such that
\begin{equation}
  M_k(N) \sim D_k \sigma_k^N \qquad (N\to+\infty).\label{eq:asympt-MkN}
\end{equation}
It is well-known~\cite[eq.~(1.4)]{BaakeCoons2018} that~$\sigma_1 = 3$ (in fact,~$M_1(N) = 3^N$ exactly). The constant~$\sigma_k$ is related to the pressure function associated to the Farey system~\cite{KesseboehmerStratmann2007,Dodds}, and one can show\footnote{This requires a slight alteration of the argument in Lemma~\ref{lem:stern-op} below, since the pressure function in~\cite{KesseboehmerStratmann2007} involves sums of~$(s(n)s(n+1))^{\tau}$ rather than~$s(n)^\tau$.} that~$\sigma_k = \exp(P(-k/2))$, where~$P(\theta)$ denotes the pressure function of the Farey system~\cite[p.135]{KesseboehmerStratmann2007}.

In Proposition 4.4.(8) of~\cite{KesseboehmerStratmann2007}, the authors show by combinatorial arguments that
$$ \phi^k \leq \sigma_k \leq \phi^k(1 + (1-\phi^{-6})^k) \qquad (\phi = \tfrac{1+\sqrt{5}}2). $$
Note that~$1-\phi^{-6} \approx 0.944\dotsb$; we also refer to~\cite[Theorem~4.15]{Dodds} for a qualitative estimate.
Also in this case we are able to identify a secondary term in the asymptotic expansion.
\begin{theorem}\label{th:stern}
  Let~$\phi=\frac{1+\sqrt{5}}2$. For~$\delta_2=\frac2{\sqrt{5}} = 0.8944\dotsb$ and~$\eta = 0.837$, we have
  $$ \sigma_k = \phi^k\big(1 + \delta_2^k + O(\eta^{k})). $$
\end{theorem}

Using a suitable uniform version of our arguments, particularly the size of the series~$\sum_r V_r^+$, $\sum_r V_r^-(x)$ in Lemma~\ref{lem:size-V} below, one could deduce an upper-bound for the number of very large values of~$s(n)$ (see~\cite{Paulin2017} for works on related questions).

\subsection{Perturbations of composition operators}\label{poco}

We will obtain Theorems~\ref{th:MMR} and~\ref{th:stern} as consequences of a more general result on perturbations of composition operators, for which we need to introduce some notation.

Let~$X$ be a set,~$a, b : X \to X$ be two maps and~$\kappa:X\to \C$ be a bounded map. We assume that~$a$ has a unique fixed point~$x_0\in X$, which is attracting on~$X$; we will assume stronger estimates below. Denote~$L^\infty(X)$ the set of bounded functions from~$X$ to~$\C$, and define~$T:L^\infty(X) \to L^\infty(X)$ by
\begin{equation}
  T[f](x) = (f\circ a)(x) + \kappa(x) (f\circ b)(x).\label{eq:expr-T-kappa}
\end{equation}
Note that for~$\kappa=0$, the operator $T_0 : f\mapsto f\circ a$ has spectral radius~$1$, and in this case~$1$ is an eigenvalue. A corresponding eigenfunction is~$\1$, with eigenprojection given by~$f \mapsto f(x_0)\1$. Define
$$ \kappa_0 := \kappa(x_0). $$
An application of~\cite[th.~VIII.2.6]{Kato} (see also Theorem~1.6 of~\cite{Kloeckner2019}) shows that if~$T_0$ is compact, if~$1$ is an isolated simple eigenvalue of~$T_0$, and if~$\|\kappa\|_\infty$ is small enough in terms of~$a$, then the spectral radius of~$T$ is asymptotically~$1 + \kappa_0 + O(\|\kappa\|_\infty^2)$. In order for this estimate to be useful, it is crucial that~$\|\kappa\|_\infty^2 = o(\kappa_0)$. The setting in which we are interested here is one where such a bound is not satisfied because~$\kappa$ does not decay uniformly in~$X$.

We will answer this question, in the special case~$\kappa\geq 0$ and under the specific conditions stated below, by constructing an approximate eigenfunction and taking into account the interaction of~$a$ and~$b$ on~$X$. For~$k_1, k_2, \dotsc \in\N_{\geq 0}$ and~$x\in X$, we will use the shorthand notation~$a^{k_1} b^{k_2} \dotsc x$ for~$(a^{k_1}\circ b^{k_2} \circ \dotsc)(x)$.

Let~$(\alpha_k^+), (\alpha^-_k), (\beta_\ell), (\delta_\ell)$ (with indices~$k, \ell\in\N_{\geq 0}$) be sequences of non-negative real numbers. Assume that~$\gamma>0$, $\beta_0\geq 1$, and
\begin{align}
  c_1 := {}& \sum_{k\geq 0}\alpha_k^+ < +\infty \qquad \sum_{\ell \geq 1} \delta_\ell \beta_1 \dotsb \beta_{\ell-1} < +\infty, \label{eq:num-1} \\
  \eta := {}& \gamma+\sum_{k\geq 0} \alpha^-_k + \sum_{\ell \geq 2} \beta_1 \dotsb \beta_{\ell-1} < +\infty, \label{eq:num-2}
\end{align}
We make the following hypotheses.
\begin{align}
  \kappa(b^\ell x) {}& \leq \beta_\ell, {}& (\ell \geq 0) \label{eq:hyp-b} \\
  0 < \kappa_0 - \alpha_k^- \leq \kappa(a^k x) {}& \leq \kappa_0 + \alpha^+_k, {}& (k\geq 0), \label{eq:hyp-a} \\
  \kappa(a^kb a x) {}& \leq \kappa_0 + \gamma \alpha^+_k, {}& (k \geq 1). \label{eq:hyp-aba}
\end{align}
Finally, let~$g : X \to \R_+^*$ be such that
\begin{equation}
  \sup_{x\in X}\Big(g(x) + \frac{1}{g(ax)}\Big) < \infty, \qquad  \sup_{x\in X} \frac{g(x)}{g(b^\ell x)} + \sup_{\substack{x, y\in X}} \frac{g(x)}{g(b^\ell a y)} \leq \delta_\ell. \label{eq:hyp-g}
\end{equation}
Let~$T_{[g]}$ act on functions on~$X$ by~$T_{[g]}[f] := g T[g^{-1}f]$ (this is well-defined by~\eqref{eq:hyp-g}).

\begin{theorem}\label{thm:main}
  Under the conditions~\eqref{eq:num-1}--\eqref{eq:hyp-g}, if~$\kappa_0$ and~$\eta$ are small enough in terms of~$c_1$, then the series
  \begin{equation}\label{eq:defF}
   F_x(z) = \sum_{r\geq 0} z^{r} T_{[g]}^r[\1](x) \qquad (x\in X), \qquad F_+(z) = \sum_{r\geq 0} z^{r} \|T_{[g]}^r[\1]\|_\infty 
   \end{equation}
  have radius of convergence~$1-\kappa_0 + O_{c_1}(\eta\kappa_0 + \kappa_0^2)$, where the implied constant depends at most on~$c_1$. In particular,
  $$ \limsup_{r\to\infty} \| T_{[g]}^r\|_\infty^{1/r} = 1 + \kappa_0 + O_{c_1}(\eta\kappa_0+\kappa_0^2). $$
\end{theorem}

Translating Theorem~\ref{thm:main} in terms of an expansion of the leading eigenvalue of~$T$, instead of the spectral radius, would \textit{a priori} require additional hypotheses on~$a$ and~$b$, at the cost of restraining the applications. In the applications mentioned above, the objects of interest are, in fact, the iterates of some fixed function.

The method could in principle be extended to provide further lower order term, under a strengthening of the condition~\eqref{eq:hyp-a}, but this is not straightforward to carry out, especially compared with the methods of~\cite{Kato,Kloeckner2019}.

\section*{Acknowledgment}

The authors wish to thank L. Spiegelhofer for discussions on the topics of this work, and the anonymous referee for suggestions which helped improve the manuscript.

S. Bettin is member of the INdAM group GNAMPA and his work is partially supported by PRIN 2017 ``Geometric, algebraic and analytic methods in arithmetic'' and by INdAM. 

\section{Proof of Theorem~\ref{thm:main}}

The proof of Theorem~\ref{thm:main} is simply based on an explicit estimation of iterates of~$T_{[g]}$. In the proof, we denote~$c_2>0$ any number satisfying
$$ \beta_0 + \sum_{\ell\geq 1} \delta_\ell \beta_1 \dotsb \beta_{\ell-1} + \sup_{x\in X}\Big(g(x) + \frac{1}{g(a x)}\Big) \leq c_2. $$
The value of~$c_2$ will not affect the uniformity of the error term.

Given a word~$w = w_1 \dotsb w_n \in\{a,b\}^*$, of length~$\abs{w} = n$, and~$x\in X$, we interpret~$wx$ to mean~$w_1 \circ \dotsb \circ w_n (x)$. Let~$\eps$ denote the empty word. For all~$w\in\{a, b\}^*$ and~$x\in X$, we define~$u(w, x)$ recursively by
\begin{equation}
u(\eps, x) = 1, \qquad u(wa, x) = \frac{g(x)}{g(ax)} u(w, ax), \qquad u(wb, x) = \frac{g(x)}{g(bx)} \kappa(x) u(w, bx).\label{eq:def-u}
\end{equation}
It is easily seen, by induction, that
\begin{equation}
u(w, x) = \frac{g(x)}{g(wx)} \prod_{\substack{v\in\{a, b\}^* \\ w \in \{a, b\}^* b v}} \kappa(v x),\label{eq:prod-u}
\end{equation}
where the product is over all words~$v$ such that~$bv$ is a suffix of~$w$. For instance,
$$ u(aba^4b^2aba, x) = \frac{g(x)}{g(aba^4b^2abax)} \kappa(a^4b^2aba x) \kappa(baba x) \kappa(abax) \kappa(ax). $$
By iterating the relations~\eqref{eq:def-u}, we obtain that for all~$r\geq 0$,
\begin{equation}
  T_{[g]}^r[\1] (x) = \ssum{w\in\{a,b\}^r} u(w, x).\label{eq:series-Tr}
\end{equation}
There are as many~$\kappa$-factors in~$u(w, x)$ as occurences of~$b$ in~$w$. Since we expect~$\kappa$ to typically have small value, the main contribution to the sum~\eqref{eq:series-Tr} is expected to come from words containing few occurences of~$b$. For these terms, we expect the product~\eqref{eq:prod-u} to consist of words~$v$ starting with a long string of~$a$, and so with an associated~$\kappa$-value close to~$\kappa_0$. Similarly, under some regularity assumptions on~$g$ (which we eventually will not need), we expect $g(wx)\approx g(x_0)$ for such words.
If~$|w|_b$ denotes the number of occurrences of~$b$ in~$w$, then we are indeed led to expect~$T_{[g]}^r[\1](x) \approx \frac{g(x)}{g(x_0)}\sum_{w\in\{a, b\}^r} \kappa_0^{|w|_b} = \frac{g(x)}{g(x_0)}(1+\kappa_0)^r$.

We seek an upper-bound for~$u(w, x)$ valid for all words~$w$, and a lower-bound valid for specific words which are expected to yield the main contribution to the sum~\eqref{eq:series-Tr}. For~$\ell\geq 1$, write
$$ \sigma_\ell = \beta_1 \dotsb \beta_{\ell-1}, \qquad \delta_{k,\ell} = \begin{cases} c_2^{2} & (k>0), \\ \delta_\ell & (k=0), \end{cases}
 \qquad \gamma_{\ell} = \begin{cases} \gamma & (\ell=1), \\ 1 & (\ell>1). \end{cases}
 $$
with the convention~$\sigma_1=1$. To ease notations, we also denote
$$ \Pi(k_0, \dotsc, k_r) = \Big(\prod_{\substack{1\leq j \leq r \\ j\text{ odd}}} \sigma_{k_j} \Big) \Big(\prod_{\substack{1 \leq j \leq r-3 \\ j\text{ odd}}} (\kappa_0 + \gamma_{k_{j+2}} \alpha^+_{k_{j+1}}) \Big). $$
\begin{lemma}\label{lem:estim-u}
  For~$r\geq 2$, $k_0\in\N_{\geq 0}$,~$k_1, \dotsc, k_r\in\N_{>0}$ and~$x\in X$, we have
  \begin{align*}
    u(a^{k_0}b^{k_1} \dotsb a^{k_r}, x) \leq {}& \delta_{k_0, k_1} (\kappa_0 + \alpha_{k_r}^+) \Pi(k_0, \dotsc, k_r) {}& (r\text{ even}) \numberthis\label{eq:u-bound-upper-even}\\
    u(a^{k_0}b^{k_1} \dotsb b^{k_r}, x) \leq {}& \delta_{k_0, k_1} \beta_0 (\kappa_0 + \alpha_{k_{r-1}}^+) \Pi(k_0, \dotsc, k_r). {}& (r\text{ odd})\numberthis\label{eq:u-bound-upper-odd}\\
    \intertext{Moreover, for~$r\geq 0$, $k_0, k_1, \dotsc, k_r \geq 0$ and~$x\in X$, we have}
    u(a^{k_0} b a^{k_1} \dotsb b a^{k_r}, x) \geq {}& c_2^{-1} g(x) \prod_{1\leq j \leq r} (\kappa_0 - \alpha^-_{k_j}). \numberthis\label{eq:u-bound-lower}
  \end{align*}
\end{lemma}

\begin{proof}
  Let us examine the case of positive, even~$r$. Then
  $$ u(a^{k_0}b^{k_1} \dotsb a^{k_r}, x) = \frac{g(x)}{g(a^{k_0} b^{k_1} \dotsb x)} \prod_{\substack{j=1 \\ \text{odd}}}^{r-1} \prod_{\ell=1}^{k_j}
  \kappa(b^{k_j-\ell} a^{k_{j+1}} \dotsb x). $$
  By~\eqref{eq:hyp-b}, we have~$\kappa(b^{k_j-\ell} a^{k_{j+1}} \dotsb x)\leq \beta_{k_j-\ell}$ if~$1\leq \ell \leq k_j-1$. If~$\ell = k_j$, then we may use~\eqref{eq:hyp-a}-\eqref{eq:hyp-aba} to obtain~$\kappa(a^{k_{j+1}} \dotsb x) \leq \kappa_0 + \gamma_{k_{j+2}} \alpha_{k_{j+3}}^+$ if~$j\leq r-3$, whereas if~$j = r-1$, then we use~\eqref{eq:hyp-a} to get~$\kappa(a^{k_r} x) \leq \kappa_0 + \alpha_{k_r}^+$. Finally, the hypotheses~\eqref{eq:hyp-g} yield~$\frac{g(x)}{g(a^{k_0} b^{k_1} \dotsb x)} \leq \delta_{k_0, k_1}$ in all cases. The proof for odd~$r$ and for the bound~\eqref{eq:u-bound-lower} is similar.
\end{proof}

We now sum this over~$r\geq 0$. Let
\begin{align*}
  S_\sigma(\rho) = {}& \sum_{\ell\geq 1} \rho^\ell\sigma_\ell, & S_+(\rho) = {}& \sum_{k\geq 1}\rho^k(\kappa_0 + \alpha_k^+), \\
  S_\delta(\rho) = {}& \sum_{\ell\geq 1} \rho^\ell \delta_\ell\sigma_\ell, & S_-(\rho) = {}& \sum_{k\geq 1}\rho^k(\kappa_0-\alpha_k^-), \\
  S_{*}(\rho) = {}& \sum_{k,\ell\geq 1} \rho^{k+\ell}\sigma_\ell(\kappa_0+\gamma_\ell\alpha_k^+).
\end{align*}
Define further
\begin{align*}
 V_r(x) &:= \ssum{k_0, \dotsc k_r \geq 0} \rho^{r + \sum_j k_j} u(a^{k_0}ba^{k_1} \dotsb b a^{k_r}, x),\\
V_r^+& := \ssum{k_0 \geq 0 \\ k_1, \dotsc, k_r \geq 1} \rho^{\sum_j k_j} \| u(a^{k_0}b^{k_1} \dotsb \ast^{k_r}, \cdot)\|_\infty, 
\end{align*}
where~$\ast = a$ or~$b$ according to whether~$r$ is even or odd. Note that, by~\eqref{eq:series-Tr} and positivity, for $F_x$ and $F_+$ as in~\eqref{eq:defF} we have
\begin{equation}\label{eq:main_ineq}
 \sum_{r\geq 0} V_r(x) \leq F_x(\rho) \leq F_+(\rho) \leq \sum_{r\geq 0} V_r^+. 
 \end{equation}

\begin{lemma}\label{lem:size-V}
  For~$0\leq\rho<1$, we have
  \begin{itemize}
    \item $V_0^+ \leq \frac{c_2^{2}}{1-\rho}$,
    \item $V_1^+ \leq \beta_0(S_\delta(\rho)+ c_2^2\frac{\rho}{1-\rho}S_\sigma(\rho))$,
    \item if~$r$ is even and~$r\geq 2$,
    $$ V_r^+ \leq ( S_\delta(\rho) + c_2^2\tfrac{\rho}{1-\rho}S_\sigma(\rho)) S_*(\rho)^{(r-2)/2} S_+(\rho), $$
    \item if~$r$ is odd and~$r\geq 3$,
    $$ V_r^+ \leq \beta_0( S_\delta(\rho) + c_2^2\tfrac{\rho}{1-\rho}S_\sigma(\rho)) S_\sigma(\rho) S_*(\rho)^{(r-3)/2} S_+(\rho), $$
    \item for all~$r\geq 0$,
    $$ V_r(x) \geq c_2^{-1} g(x) \tfrac1{1-\rho}(\rho S_-(\rho))^r. $$
  \end{itemize}
\end{lemma}

\begin{proof}
The first two inequalities follow easily as in the proof of Lemma~\ref{lem:estim-u}. Moreover, if $r\geq2$ is even, summing the estimate~\eqref{eq:u-bound-upper-even} we have
\begin{align*}
    V_r^+ &\leq \ssum{k_0 \geq 0 \\ k_1 \geq 1} \rho^{k_0+k_1} \delta_{k_0, k_1}\sigma_{k_1} \ssum{k_r \geq 1} \rho^{k_r} (\kappa_0 + \alpha_{k_r}^+)  \prod_{\substack{j=1 \\ j\text{ odd}}}^{r-3}\ssum{k_{j+1} \geq 1,\\ k_{j+2}\geq1}\rho^{k_{j+1}+k_{j+2}}\sigma_{k_{j+2}}(\kappa_0 + \gamma_{k_{j+2}} \alpha^+_{k_{j+1}})\\
   &= ( S_\delta(\rho) + c_2^2\tfrac{\rho}{1-\rho}S_\sigma(\rho))S_+(\rho) S_*(\rho)^{(r-2)/2}.
\end{align*}
The last two inequalities can be obtained in a similar way.
\end{proof}

We are ready to prove Theorem~\ref{thm:main}. On the one hand, we deduce that~$\sum_{r\geq 0} V_r^+$ converges if~$\rho<1$ and~$S_*(\rho) < 1$. 
But by~\eqref{eq:num-2} and the definition of $S_*$,
\begin{equation}
  S_*(\rho) \leq \frac{\kappa_0(1+\eta)}{1-\rho} + c_1\eta,\label{eq:majo-Sast}
\end{equation}
so that~$S_*(\rho)<1$ if~$\rho \leq 1 - \kappa_0 - c'\eta\kappa_0$, for some real number~$c'$ depending on~$c_1$. We conclude that the radius of convergence~$\rho_+$ of~$F_+(z)$ satisfies~$\rho_+ \geq 1 - \kappa_0 + O_{c_1}(\eta\kappa_0)$. 

On the other hand, we deduce that~$\sum_{r\geq 0} V_r(x)$ diverges if~$\rho S_-(\rho)>1$. Since
$$ S_-(\rho) \geq \frac{\kappa_0 \rho}{1-\rho} - \eta, $$
we deduce that~$\rho S_-(\rho) > 1$ if~$\rho \geq 1 - \kappa_0 + c'(\eta\kappa_0 + \kappa_0^2)$ if~$c'$ is taken large enough. We conclude that the radius of convergence~$\rho(x)$ of~$F_x(z)$ satisfies~$\rho(x) \leq 1-\kappa_0+O(\eta\kappa_0+\kappa_0^2)$. Theorem~\ref{thm:main} then follows by~\eqref{eq:main_ineq}.


\section{Proof of Theorem~\ref{th:MMR}}

For all~$x\in[0, 1]$, define
$$ a(x) = 1-\frac x2, \qquad b(x) = \frac x2, \qquad S(x) = \frac2{\sqrt{3}}\sin\Big(\frac {\pi x}2\Big). $$
Note that
\begin{equation}
  a^n(x) = \frac23 + \Big(\frac{-1}2\Big)^n\Big(x-\frac23\Big), \qquad b^n(x) = \frac1{2^n} x,\label{eq:cor1-iterates-ab}
\end{equation}
and~$S(2/3)=1$. Therefore, the product
\begin{equation}
  G(x) = \prod_{n\geq 0} S(a^n x)\label{eq:prod-g}
\end{equation}
converges absolutely for~$x\in(0, 1]$; note that, due to the~$n=0$ term, it vanishes at order~$1$ at~$x=0$. Finally, let~$\tau>0$, and~$g, \xi , \kappa :[0,1]\to[0,1]$ be given by
$$ \xi(x) := \frac{G(x/2)}{G(1-x/2)}, \qquad g(x) := G(x)^\tau, \qquad \kappa(x) := \xi(x)^\tau $$
The functions~$G$ and~$\xi$ are depicted in Figure~\ref{fig:plots_g_xi}.

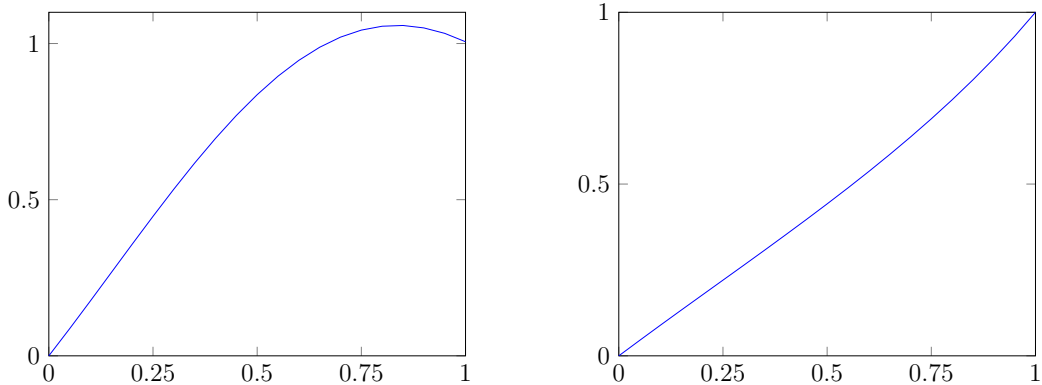
\begin{figure}[h!]
  \begin{subfigure}[t]{0.45\textwidth}
    \centering
    \begin{tikzpicture}[scale=0.8]
      \begin{axis}[
        xmin = 0, xmax = 1,
        ymin = 0, ymax = 1.1,
        xtick = {0, 0.25, 0.5, 0.75, 1},
        ytick = {0, 0.5, 1}]
        \addplot[color = blue]
        coordinates{(0.0, 0.0)(0.05, 0.087)(0.1, 0.1763)(0.15, 0.2668)(0.2, 0.3573)(0.25, 0.4466)(0.3, 0.5337)(0.35, 0.6173)(0.4, 0.6962)(0.45, 0.7696)(0.5, 0.8362)(0.55, 0.8954)(0.6, 0.9462)(0.65, 0.9881)(0.7, 1.0205)(0.75, 1.043)(0.8, 1.0554)(0.85, 1.0578)(0.9, 1.05)(0.95, 1.0326)(1.0, 1.0057)};
      \end{axis}
    \end{tikzpicture}
  \end{subfigure}
  ~
  \begin{subfigure}[t]{0.45\textwidth}
    \centering
    \begin{tikzpicture}[scale=0.8]
      \begin{axis}[
        xmin = 0, xmax = 1,
        ymin = 0, ymax = 1,
        xtick = {0, 0.25, 0.5, 0.75, 1},
        ytick = {0, 0.5, 1}]
        \addplot[color = blue]
        coordinates{(0.0, 0.0)(0.05, 0.0449)(0.1, 0.0893)(0.15, 0.1332)(0.2, 0.1768)(0.25, 0.2203)(0.3, 0.2638)(0.35, 0.3076)(0.4, 0.3519)(0.45, 0.3967)(0.5, 0.4423)(0.55, 0.4888)(0.6, 0.5366)(0.65, 0.5859)(0.7, 0.6369)(0.75, 0.6899)(0.8, 0.7453)(0.85, 0.8035)(0.9, 0.865)(0.95, 0.9302)(1.0, 1.0)};
      \end{axis}
    \end{tikzpicture}
  \end{subfigure}
  \caption{Approximate plots of~$G$ (left) and~$\xi$ (right)}  
  \label{fig:plots_g_xi}
\end{figure}

\begin{lemma}\label{lem:lambda}
  The function~$\xi:[0, 1]\to[0, 1]$ is of~$C^1$ class, increasing and bijective.
\end{lemma}
\begin{proof}
  The values~$\xi(0)=0$ and~$\xi(1)=1$ are simple to compute. The $C^1$ regularity of~$\xi$ follows by the uniform convergence of the product defining~$G$. To see that~$\xi'>0$, we define, for all~$x\in[0, 1]$ and~$n\geq 0$, with~$x\neq 0$ if~$n=0$,
  $$ h_n(x) = \cot\Big(\frac\pi3 + \frac\pi2 \Big(\frac{-1}2\Big)^{n}\Big(\frac x2 - \frac23\Big)\Big) + \cot\Big(\frac\pi3 + \frac\pi2 \Big(\frac{-1}2\Big)^{n}\Big(\frac 13 - \frac x2\Big)\Big). $$
  By the derivative~$\cot' = -1-\cot^2$ and since~$\cot\geq0$ on~$(0, \pi/2]$, we find~$h_n' \leq 0$. Moreover, we have
  $$ h_{2n}(1) - h_{2n+1}(0) = \cot\Big(\frac\pi3-\frac\pi{12}\frac1{4^n}\Big) - \cot\Big(\frac\pi3+\frac\pi{6}\frac1{4^n}\Big) > 0. $$
  We deduce that for all~$x, y \in (0, 1]$, we have~$h_{2n}(x) > h_{2n+1}(y)$, and so
  $$ \frac{\xi'}{\xi}(x) = \frac{\pi}4 \sum_{n\geq 0} \Big(\frac{-1}{2}\Big)^n h_n(x) > 0. $$
\end{proof}

We define the operator $T_\tau:L^\infty((0,1]) \to L^\infty((0,1])$ by
\begin{equation}
  T_\tau[f] := (f \circ a) +  \kappa \cdot (f\circ b). \label{eq:link-TU}
\end{equation}

\begin{lemma}\label{lem:U-T}
  For~$k\geq 1$, we have~$\varrho_k = \frac{3^k}2 \lim_{r\to+\infty} \|g T_{2k}^r[g^{-1}]\|_\infty^{1/r}$.
\end{lemma}
\begin{remark}
  Note that the operator~$f \mapsto g T_{2k}[g^{-1} f]$ is well-defined also as an operator acting on~$C([0, 1])$, since~$g\circ a > 0$ on~$[0, 1]$ and by extending~$\frac{g}{g\circ b}$ continuously at~$0$.
\end{remark}

\begin{proof}
  By Proposition~1 of~\cite{MauduitEtAl2018}, we have
  $$ \varrho_k = \lim_{r\to+\infty} \| P_k^r[\1]\|_\infty^{1/r}, $$
  where~$P_k$ acts on continuous functions on~$[0,1]$ by
  $$ P_k[f](x) = \frac12\Big(2\sin\Big(\frac{\pi x}2\Big)\Big)^{2k} f\Big(\frac x2\Big) + \frac12\Big(2\cos\Big(\frac{\pi x}2\Big)\Big)^{2k} f\Big(\frac{x+1}2\Big). $$ 
  Note that~$P_\tau$ preserves the subspace of functions symmetric with respect to $\frac12$. We ``desymmetrize'' it by defining, for all~$\tau>0$, an operator~$U_\tau$ on~$C([0, 1])$ by
  $$ U_\tau[f](x) = S(x)^{\tau} \Big(f\Big(1-\frac x2\Big) + f\Big(\frac x2\Big)\Big). $$
Then, writing $f^*(t):=f(1-t)$, we have
  \begin{align*}
    P_k[f+f^\ast](x) = {}& \frac{3^k}{2}\big(U_{2k}[f](x) + U_{2k}[f](1-x)\big) \\
    = {}& \frac{3^k}2\big(U_{2k}[f](x) + U_{2k}[f]^\ast(x)\big),
  \end{align*}
   and so, by induction, we have for all~$r\in\N$
   \begin{equation} \label{relpu}
     P_k^r[f+f^*](x) = \Big(\frac{3^k}2\Big)^r(U_{2k}^r[f](x) + U_{2k}^r[f]^\ast(x)). 
   \end{equation}
   We take $f=\1$, and deduce by positivity that $\frac12\|U_{2k}^r[\1]\|_\infty\leq\|P_k^r[\1]\|_\infty\leq \|U_{2k}^r[\1]\|_\infty$. In particular,
  \begin{equation}
    \varrho_k = \frac{3^k}2 \lim_{r\to+\infty} \| U_{2k}^r[\1]\|_\infty^{1/r}.\label{eq:link-rho-Ur}
  \end{equation}
  By construction, we have~$T_\tau[f] = g^{-1} U_\tau[gf]$ for all~$f\in C((0, 1])$, in other words,~$gT_\tau[g^{-1}f] = U_\tau[f]$. This yields the claimed formula.
\end{proof}

We can now finish the proof of Theorem~\ref{th:MMR}.
Since~$G(x)$ vanishes at order~$1$ at~$x=0$, we may find~$c>0$ so that~$(cx)^\tau \leq g(x) \leq (x/c)^\tau$. Also, note that for~$0\leq x \leq y \leq 1$,
\begin{equation}\label{eq:lips}
 \xi(y)^\tau - \xi(x)^\tau \leq (y-x) \|\xi'\|_\infty \tau \xi(y)^{\tau-1}.
 \end{equation}
Define~$\kappa_0 := \xi(2/3)^\tau$, and
\begin{align*}
  \beta_\ell = {}& \xi(2^{-\ell})^\tau, \\
  \alpha_k^- = {}& \tfrac23 2^{-k} \|\xi'\|_\infty \tau \xi(\tfrac56)^{\tau-1}, \\
  \alpha_k^+ = {}& \begin{cases} 1 & (k\in \{0, 1\}), \\ 2^{-k}  \max(1, \tfrac23 \|\xi'\|_\infty \tau \xi(\tfrac34)^{\tau-1}) & (k\geq 2), \end{cases} \\
  \gamma ={}&  \tfrac23 \|\xi'\|_\infty \tau \xi(\tfrac78)^{\tau-1}, \\
  \delta_\ell = {}& c^{-2\tau}2^{1+(\ell+1) \tau}.
\end{align*}
We apply Theorem~\ref{thm:main} with~$\kappa = \xi^\tau$.
The condition~\eqref{eq:hyp-b} follows from the fact that~$\xi$ is increasing, and~$b^\ell([0, 1]) = [0, 2^{-\ell}]$. The condition~\eqref{eq:hyp-a} follows from~\eqref{eq:lips} and the inclusion~$a^k[0, 1] \subset [\frac23(1-2^{-k}), \frac23(1 + 2^{-k})]$. The condition~\eqref{eq:hyp-aba} follows from the inclusion~$a^kba[0, 1] \subset [0, \frac78]$, and the condition~\eqref{eq:hyp-g} follows from~$a [0, 1] \subset[\tfrac12, 1]$. The convergence of the series~\eqref{eq:num-1} is ensured by the fact that~$\xi(2^{-\ell}) \to 0$ as~$\ell\to\infty$. With~$\eta = O(\tau \xi(\frac78)^\tau)$, the above yields
$$ \limsup_{r\to+\infty} \|g T_\tau^r[g^{-1}]\|_\infty^{1/r} = 1 + \kappa_0 + O(\eta \kappa_0). $$
Lemma~\ref{lem:U-T} finishes the proof of Theorem~\ref{th:MMR}. From Lemma~\ref{lem:lambda}, we have~$\xi(\frac78) < 1$; the more precise bound~$\xi(\frac78) \in[0.833, 0.835]$ is checked numerically by truncating the product~\eqref{eq:prod-g} at~$n=11$ and estimating the remainder.

\section{Proof of Theorem~\ref{th:stern}}\label{sec:proof-theor-stern}

For~$x\in[0, 1]$, let
$$ a(x) = \frac1{1+x}, \qquad b(x) = \frac{x}{1+x}, $$
and for all~$\tau\geq 0$, define
$$ g(x) = (\phi + x)^\tau, \qquad \xi(x) = \frac{1 + \phi x}{\phi + x}, \qquad \kappa(x) = \xi(x)^\tau. $$
Note that $\xi$ is an increasing function with $\xi(0)=\phi^{-1}$, $\xi(1)=1$.
It is easy to see that if~$(F_n)_{n\geq 0} = (0, 1, 1, \dotsc)$ denotes the Fibonacci sequence, then for all~$n\in\N_{\geq 1}$,
$$ a^{ n}(x) = \frac{F_{n-1} x + F_{n}}{F_{n} x + F_{n+1}}, \qquad b^n(x) =\frac{x}{1+nx}. $$
Note also that the map~$\kappa:[0, 1]\to [0, 1]$ is increasing, with~$\kappa(1) = 1$.

For notation convenience, the variable~$k$ in the statement of Theorem~\ref{th:stern} will be renamed~$\tau$. In this section,~$\tau$ is a positive integer.

We define an operator~$T_\tau$ on~$C([0, 1])$ by
$$ T_\tau[f] = (f \circ a) + \kappa \cdot (f \circ b). $$
\begin{lemma}\label{lem:stern-op}
  For all~$\tau\in\N_{>0}$, there exist constants~$\sigma_\tau, D_\tau>0$ such that the asymptotic formula~\eqref{eq:asympt-MkN} holds. Moreover, we have
  $$ \sigma_\tau = \phi^{\tau} \limsup_{r\to \infty} \| g T_\tau^r[g^{-1}]\|_\infty^{1/r}, $$
\end{lemma}
\begin{proof}
  We claim that for all~$N\geq 1$,
  \begin{equation}
    M_\tau(N) - M_\tau(N-1) = \tfrac12 P_\tau^{N}[\1](1),\label{eq:rel-Mk-Pk}
  \end{equation}
  where~$P_\tau$ acts on degree~$\tau$ polynomials by
  $$ P_\tau[f](x) = (1+x)^\tau\Big(f\Big(\frac1{x+1}\Big) + f\Big(\frac{x}{x+1}\Big)\Big). $$
  To prove this, we let $B_0 =  \smatrix{0 & 1 \\ 1 & 1}$ and $B_1 =  \smatrix{ 1 & 0 \\ 1 & 1}$. Then, by the chain rule~\cite[eq.~(2.3)]{Iwaniec1997}, it follows that
  \begin{equation}
    P_\tau^N[f](x) = \ssum{\eps_0, \dotsc, \eps_{N-1} \in \{0, 1\} \\ M = B_{\eps_0} \dotsb B_{\eps_{N-1}}} j_M(x)^\tau f(M\cdot x)\label{eq:sum-PkN-Bj}
  \end{equation}
  where~$j_M(x) = cx + d$ if~$M = \smatrix{a & b \\ c & d}$. We now recall that if~$2^N\leq n < 2^{N+1}$ is written~$n = 2^N + \sum_{0\leq j < N} \eps_j 2^j$ in base~$2$, then the formula~\cite[eq.~(2.1)]{BDS}
  \begin{equation}
    \begin{pmatrix} s(n+1) \\ s(n) \end{pmatrix} = A_{\eps_0} \dotsb A_{\eps_{N-1}} \begin{pmatrix} 1 \\ 1 \end{pmatrix},
  \end{equation}
  holds, where~$A_0 = \smatrix{1 & 1 \\ 0 & 1}$ and $A_1 =  \smatrix{1 & 0 \\ 1 & 1}$. We wish to rewrite the sum~\eqref{eq:sum-PkN-Bj} in terms products of~$A_0$ and~$A_1$. Let~$T =  \smatrix{0 & 1 \\ 1 & 0}$, so that $TA_0 = A_1T = B_0$, and also~$TA_1=A_0T$. To each tuple~$(\eps_0, \dotsc, \eps_{N-1}) \in \{0, 1\}^N$, we associate a tuple~$(\eps'_0, \dotsc, \eps_N')\in\{0, 1\}^{N+1}$ such that
  $$ M := B_{\eps_0} \dotsb B_{\eps_{N-1}} = A_{\eps'_0} \dotsb A_{\eps'_{N-1}} T^{\eps'_N}, $$
  by writing~$B_1=A_1$, $B_0 = TA_0$, and then pushing all the occurences of~$T$ to the right, using~$T^2 = {\rm id}$ and~$TA_0 = A_1 T$. Then~$\eps'_N$ is given by the sign of~$\det(M)$. We also always have~$\eps'_0 = 1$. Finally, the map~$(\eps_0, \dotsc, \eps_{N-1}) \mapsto (\eps'_1, \dotsc, \eps'_N)$ is injective, since~$B_0$ and~$B_1$ are free over~$GL_2(\N)$ (their transposes map~$(\R_+^*)^2$ into~$\{(x, y)\in\R^2, x>y>0\}$ and~$\{(x, y)\in\R^2, y>x>0\}$ respectively), and thus also bijective.
  Using this bijection in~\eqref{eq:sum-PkN-Bj}, we deduce
  $$ P_\tau^N[\1](1) = \ssum{\eps'_1, \dotsc, \eps'_N \in \{0, 1\} \\ M = A_1 A_{\eps'_1} \dotsb A_{\eps'_{N-1}} T^{\eps'_N}} j_M(1)^\tau. $$
  Now we note that~$T\cdot 1 = 1$, so that for each tuple~$(\eps'_1, \dotsc, \eps'_N)$ in the sum,
  \begin{align*}
    j_M(1) = {}& \ppmatrix{0 & 1} A_1 A_{\eps'_1} \dotsb A_{\eps'_{N-1}} T^{\eps'_N} \ppmatrix{1 \\ 1} \\
    = {}& \ppmatrix{0 & 1} A_1 A_{\eps'_1} \dotsb A_{\eps'_{N-1}} \ppmatrix{1 \\ 1} \\
    = {}& s(n'),
  \end{align*}
  where~$s' = 2^N + \sum_{1\leq j < N} \eps'_j 2^j + 1$. Note that this is independent of~$\eps'_N$. As~$(\eps'_1, \dotsc, \eps'_{N-1})$ runs through~$\{0,1\}^{N-1}$,~$n'$ runs through the odd integers in~$[2^N, 2^{N+1})$. We deduce that
  $$ P_\tau^N[\1](1) = 2\ssum{2^N \leq n < 2^{N+1} \\ n\text{ odd}} s(n)^\tau, $$
  and finally~\eqref{eq:rel-Mk-Pk} follows since~$s(2n) = s(n)$.
  
  Since~$P_\tau$ acting on the set~$\R_\tau[x]$ of real polynomials of degree~$\leq \tau$, with its canonical basis, is positive, by the Perron-Frobenius theorem, it has a simple isolated dominant eigenvalue~$\sigma_\tau>0$, equal to its spectral radius, and actually~$\sigma_\tau> 1$ since~$P[\1] \geq 2\cdot \1$. We have in particular, by positivity,
  $$ \sigma_\tau = \limsup_{r\to\infty} \|P_\tau^r[\1]\|_\infty^{1/r}. $$
  By  spectral decomposition, we deduce the existence of a constant~$D'_\tau>0$ such that, as~$N\to \infty$,
  $$ M_\tau(N) - M_\tau(N-1) \sim D'_\tau \sigma_\tau^N, $$
  and therefore~$M_\tau(N) \sim D_\tau \sigma_\tau^N$ with~$D_\tau = D'_\tau \sigma_\tau/(\sigma_\tau-1)$. To conclude the proof, it suffices to remark that, by construction, $P_\tau[f] = \phi^\tau g T_\tau[g^{-1}f]$.
\end{proof}

Note that~$\|(a\circ a)'\|_\infty \leq 1/2$ and~$a(\frac1\phi) = \frac1\phi$, so that for~$k\in\N$,
$$ \| a^k - \tfrac1\phi\|_\infty \leq 2^{1-k/2}. $$
We let~$\kappa_0 := \kappa(\tfrac1\phi) = (\frac2{\sqrt{5}})^\tau$. Define
\begin{align*}
  \beta_\ell = {}& \xi(1/(\ell+1))^\tau \\
  \alpha^-_k = {}& 2^{1-k/2} \|\xi'\|_\infty \tau \xi(\tfrac1\phi)^{\tau-1}, \\
  \alpha_k^+ = {}& \begin{cases} 1 & (k\in \{0, 1\}), \\ 2^{1-k/2}  \max(1, \|\xi'\|_\infty \tau \xi(\tfrac23)^{\tau-1}) & (k\geq 2), \end{cases} \\
  \gamma ={}&  \|\xi'\|_\infty \tau \xi(\tfrac34)^{\tau-1}, \\
  \delta_\ell = {}& \phi^\tau
\end{align*}
We apply Theorem~\ref{thm:main}.
The hypothesis~\eqref{eq:hyp-b} is satisfied since~$b^\ell[0, 1] = [0, \tfrac1{\ell+1}]$.
The hypothesis~\eqref{eq:hyp-a} is satisfied by the inclusion~$a^k[0, 1] \subset [0, \tfrac23]$ if~$k\geq 2$.
The hypothesis~\eqref{eq:hyp-aba} follows by the inclusion~$a^kba[0, 1] = a^k[\tfrac13, \tfrac12] \subset [0, \tfrac34]$.
Finally the hypothesis~\eqref{eq:hyp-g} follows from~$\phi^\tau \leq g(x) \leq \phi^{2\tau}$.
We obtain~$\eta \ll \tau\xi(\tfrac34)^\tau$ and~$c_1$ bounded independently of~$\tau$, and deduce
$$ \limsup_{r\to \infty} \| g T_\tau^r[g^{-1}]\|_\infty^{1/r} = 1 + \kappa_0 + O(\eta\kappa_0). $$
Theorem~\ref{th:stern} then follows by Lemma~\ref{lem:stern-op}.

\bibliographystyle{plain}
\bibliography{op-approx.bib}

\end{document}